\documentclass[12pt]{article}

\usepackage[bookmarks={true},bookmarksopen={true},colorlinks,citecolor = green,anchorcolor = blue,linkcolor = blue]{hyperref}
\usepackage[utf8]{inputenc}
\usepackage[T1]{fontenc}
\usepackage[english]{babel}
\usepackage[dvipsnames]{xcolor}
\usepackage{lmodern}
\usepackage{array}
\usepackage{amsmath}
\usepackage{amssymb}
\usepackage{amsfonts}
\usepackage{amsthm}
\usepackage{mathrsfs}
\usepackage{enumitem}
\usepackage{bbm}
\usepackage{hhline}
\usepackage{fullpage}
\usepackage{tikz}
\usepackage{float}
\usepackage{authblk}
\usepackage{bm}
\usetikzlibrary[topaths]
\usepackage{booktabs}
\numberwithin{equation}{section}

\newcommand{\SSN}{\mathcal S_{[N]}^{++}}

\newcommand{\LF}{\hat{\Phi}_n(L)}

\newcommand{\TL}{\Tilde{L}_n}

\makeindex
\DeclareMathOperator{\di}{\mathbf{d}}
\DeclareMathOperator{\Tr}{Tr}

\newcommand{\cD}{\mathcal{D}}

\newcommand{\cF}{\mathcal{F}}

\newcommand{\cP}{\mathcal{P}}

\newcommand{\cY}{\mathcal{Y}}

\newcommand{\bY}{\mathbf{Y}}

\newcommand{\bbP}{\mathbb{P}}

\newcommand{\YY}{\mathbb{Y}}

\newcommand{\PP}{\mathbb{P}}
\newcommand{\defeq}{\mathrel{\mathop:}=}

\newcommand{\E}{\mathbb{E}}
\newcommand{\argmax}{\arg\!\! \max}

\newcommand*\diff{\mathop{}\!\mathrm{d}}

\newcommand{\hp}{\hat{p}}

\theoremstyle{plain}
\newtheorem{proposition}{Proposition}[section]
\newtheorem{lemma}[proposition]{Lemma}
\newtheorem{definition}[proposition]{Definition}
\newtheorem{defin/prop}[proposition]{Definition/Proposition}
\newtheorem{theorem}[proposition]{Theorem}

\title{Asymptotic properties of maximum likelihood estimators for determinantal point processes}
\author[1]{\sc  Yaozhong Hu\thanks{Supported by an NSERC Discovery grant and a centennial  fund from University of Alberta.  
Email: Yaozhong@ualberta.ca  }}
\author[2]{\sc  Haiyi Shi \thanks{Corresponding author. Email: haiyi\_shi@sfu.ca}}
\affil[1]{ Department of Mathematical and Statistical Sciences, University of Alberta, Edmonton, AB, Canada.}

\affil[2]{Department of Statistics and Actuarial Science, Simon Fraser University, Burnaby, BC, Canada.}

\date{} 

\begin{document}
 \maketitle
\begin{abstract}
%Determinantal point processes (DPPs) arise as important tools in various aspects of mathematics, such as stochastic processes, random matrices, and combinatorics. Over the last decade, DPPs have also been widely used in machine learning community. They are especially popular in subset selection problems, for they favour subsets of high quality and diversity. These applications motivate studies in parameter estimations, of which a common method is the  maximum likelihood estimation. 
We obtain  the almost sure strong consistency  and the Berry-Esseen type bound for    the maximum likelihood estimator   $\Tilde L_n$ of the ensemble $L^\star$ for determinantal point processes (DPPs),  strengthening  and completing   previous work
initiated in   Brunel, Moitra, Rigollet,  and Urschel \cite{brunel2017maximum}. Numerical algorithms of estimating DPPs are developed and simulation studies are performed. Lastly, we give explicit formula and a detailed discussion for the maximum likelihood estimator for 
blocked determinantal matrix of two by two submatrices  and compare it  with the frequency 
method.
%first studied this non-convex optimization problem using an information geometric approach. Inspired by their work, we introduce and extend some of their results: we exhibit the strong consistency and the rates of convergence of the maximum likelihood estimator to the normality, i.e. the Berry-Esseen type theorem. Moreover, in two dimensional case, we obtain the explicit form of the estimator and establish the strong consistency and central limit theorem. We also give some remarks on higher dimensional DPPs. 
\end{abstract}

\section{Introduction}
Determinantal point processes (DPPs) arise from random matrix theory \cite{ginibre1965statistical} and are first introduced to give the probability distribution of Fermionic system in thermal equilibrium in quantum physics \cite{Mac75}. 
%The Pauli exclusion principle states that no two fermions can occupy the same quantum state, which leads to the so-called ``anti-bunching" effect of fermions system. 
Since then, DPPs have been found in various aspects of mathematics, including for example, loop-free Markov chains \cite{borodin2010adding},  edges of uniformly spanning trees \cite{burton1993local} and in particular, recently to image processing \cite{image}. 

In the seminal work   \cite{KT12},    Kulesza and Taskar show that DPPs demonstrate 
the unique characteristics comparing to   various other probabilistic models in the sense that they capture the
global repulsive behavior between items, give polynomial-time algorithms for statistical inference, and have geometrical intuition. Due to these advantages   DPPs have  played very important roles in    machine learning, especially in subset selection problems, such as documentary summarization, image search, and pose  determination \cite{KT12}, and 
so on. These real
 world applications necessitate the estimation of parameters of determinantal
point process models. In this context, maximum likelihood estimator $\tilde L_n$ of the true parameter $L^\star$  is a natural
choice, which in general leads to a non-convex optimization problem in our situation. Along
this direction, Kulesza and Taskar split DPPs model into diversity part and quality part and only learn the quality part while the first part is fixed. They conjecture  that the problem of learning the likelihood of DPPs is NP-hard, which has been proven by \cite{pmlr-v178-grigorescu22a}   a decade later. Brunel, Moitra, Rigollet,  and Urschel \cite{brunel2017maximum} first study the local geometry of the expected maximum likelihood estimation of DPPs, that is, the curvature of likelihood function around its maximum. Then they prove that the maximum likelihood estimator converges to true values in probability and establish the corresponding central limit theorem. Motivated by these works,  
the present  paper contains three further contributions.
\begin{enumerate}
	\item[(i)] For the consistency we improve the convergence in probability of $\tilde L_n$ to  $L^\star$ obtained in \cite{brunel2017maximum} to   almost sure convergence. This  almost sure convergence property is important in applications since in reality we only observe a single sample path. Theoretically, it is well-known that convergence almost surely is always a hard problem. It is also the case here. We get around this difficulty  by making use of  the Wald's consistency theorem in our situation.  
	\item[(ii)] The asymptotic normality of the maximum estimator is also obtained in \cite{brunel2017maximum}. This means that 
	$\sqrt{n} (\tilde L_n-L^\star)$ is   approximately a normal random variable. We study the rate of convergence associated this normality approximation. Namely, we obtain a Berry-Esseen type bound for the maximum likelihood estimator $\tilde L_n$.  
	\item[(iii)] Due to non-identifiability of DPPs  observed   in \cite{kulesza2012learning}, there is an involvement of $\hat D_n$ in the definition of $\tilde L_n$. In general case this involvement of $\hat D_n$ seems unavoidable.  To get rid of this $\hat D_n$ we consider  special cases  of blocked two by two matrix ensembles, where we obtain explicit form  of the maximum likelihood estimator,
	and consequently, no involvement of $\hat D_n$ in these cases. This  explicit form makes our 
	computations much efficient. When we apply  various numerical methods such as Newton-Raphson,
	Stochastic Gradient Descent method (SGD) to general maximum likelihood estimator we can hardly 
	obtain satisfactory results.  
\end{enumerate}   

The paper is organized as follows. In Section 2 we introduce some basic definitions and recall 
some properties of DPPs.
We present a proof of  a well-known result about the characterization of the determinantal process
(Theorem 2.2 in next section),  whose proof
we cannot find explicitly in literature.  In Section 3 we present   our main results for the almost sure consistency and the Berry-Esseen type theorem for the maximum likelihood estimator.   Newton-Raphson and stochastic gradient descent (SGD) algorithm are  suggested to find the maximum likelihood estimator in general case in Section 4, and  the results are reported although they are not satisfactory.  In Section 5, we 
discuss the explicit maximum likelihood estimator  for  the two by two blocked ensembles and since the explicit 
form is available we get very satisfactory numerical results. 
Some concluding remarks are given in Section 6.

\section{Preliminary}

 In this paper, we focus on the   finite state point process.  Let   the ground set  $\cY = [N]  :=  \{1,2, \cdots ,N\}$ be a set of $N$ points.
The set of all subsets of $\cY$ is denoted by $\YY$. 
Let $(\Omega, \cF, \PP)$ be a complete probability space
on which the expectation is denoted by $\E$. 
A (finite) point process  on a ground set $\cY$ is 
a random variable $\bY: \Omega\to \YY$.  
We also call the 
  probability measure $\cP$ for $\bY$ the point process over the subsets
of  $\cY$. Random subsets drawn from the point process $\cP$ can be any subset between null set and full set $\cY$.

\begin{definition}
     A  $\YY$-valued point process $ \bY $ is called a determinantal point process if   there is a symmetric and positive definite   $ N \times N $ symmetric matrix  matrix $K$   such that   for every fixed set $A \subseteq \cY$,
\begin{equation}
\bbP(A \subseteq \bY) = \det(K_{A})\,,  \label{def}
\end{equation}
where $K_{A}$ denotes the restriction of  
%\footnote{In general, K need not be symmetric. We assume this for simplicity.}
$K$ to %entries indexed by the elements of 
the subset A, that is, $K_{A} \defeq [K_{i,j}]_{i,j \in A}$. 
 \end{definition}
% If we think of each of item in the ground set $\cY$ as the Boolean variable, the left side of \eqref{def} is the marginal probability in certain sense and hence $K$ is called marginal kernel.   

 To our best knowledge the following  sufficient and necessary condition for a symmetric kernel $K$ to define a determinant point process is often mentioned without directly proving it so we provide a proof as follows.

\begin{theorem} 
    The sufficient and necessary condition for a symmetric kernel $K$ to define a determinant point process is $ 0 \preceq K \preceq I $.
\end{theorem}

\begin{proof} 
     Since the marginal probability of empty set is the total probability space, $\bbP(\Omega) = \bbP(\emptyset \subseteq \bY) =1$. We set $\det(K_{\emptyset}) =1$. Because $\PP$ is a probability measure, all principal minors of $K$, i.e. $\det(K_{A})$ must be nonnegative, and thus K itself must be positive semidefinite, that is, $K \succeq 0 $.
     
     Furthermore, from $\bbP(\emptyset = \bY) + \bbP(\bigcup_{i=1}^{N}\{i \in \bY \} ) = 1 $  and using inclusion–exclusion principle we get
    \begin{eqnarray}
     \bbP(\bigcup_{i=1}^{N}\{i \in \bY \} ) &=& \sum_{i \in [N]} \bbP(i\in \bY) -  \sum_{\{i,j\} \subset [N]} \bbP(\{i,j\}\subseteq \bY) + \cdots  \nonumber \\ 
     &&\qquad \cdots + (-1)^{N-1} \bbP( [N]\subseteq\bY)  \nonumber \\
     &=& \sum_{|A| =1} \det(K_{A}) - \sum_{|A| =2}\det(K_{A}) + \cdots  \nonumber \\
     &&\qquad \cdots   + (-1)^{N-1} \det(K) \nonumber \\
     &=& 1- \det(I-K)\,.  \label{inclusion}
     \end{eqnarray}

The above last equality follows from the characteristic polynomial. Equation \eqref{inclusion} also means
\begin{equation}
    \bbP(\emptyset = \bY) = \det(I-K) \geq 0. \label{empty}
\end{equation}
Similarly, we are able to show that  $\bbP(\emptyset = \bY \cap A) = \det(I_{A}-K_{A}) \geq 0$ for any subset $A \subseteq [N]$ and hence $K \preceq I$. So the necessary condition for a symmetric matrix to give a determinantal process is $ 0 \preceq K \preceq I$. 

% In particular,  all the diagonal elements of the marginal kernel $K_{i,i}$ should be in the interval $[0,1]$. We can assume $K_{i,i}$ is always greater than 0, otherwise the element $i$ can be excluded from the model. 

This condition turns out to be sufficient: any $0 \preceq K \preceq I $ defines a DPP. To prove this, it's sufficient to show that for every $A \subseteq [N]$, the atomic probability is well-defined, that is, $0 \leq\bbP(A = \bY) \leq 1$. The probability being less or equal to 1 holds since $K \preceq I $. For the other inequality, we assume $K_{A}$ is invertible.\footnote{if $K_{A}$ is not invertible, we immediately get $\bbP(A = \bY) = 0$.} Then using Schur complement and characteristic polynomial, we have
\begin{eqnarray}
    \bbP(A = \bY) &=& \bbP(A \subseteq \bY) - \bbP(\bigcup_{i\in \Bar{A} }\{A\cup\{i\} \subseteq \bY \}) \nonumber \\
    &=& \det(K_{A}) - \sum_{i \in \Bar{A}} \det(K_{A\cup\{i\}}) + \sum_{\{i,j\} \subseteq \Bar{A}} \det(K_{A\cup\{i,j\}})+\nonumber \\
    &&\qquad \cdots + (-1)^{|\Bar{A}|} \det(K) \nonumber \\
    &=& \det(K_{A}) - \sum_{i \in \Bar{A}} \det(K_{A}) \det(K_{ii} - K_{\{i\},A} K^{-1}_{A} K_{A,\{i\}} ) \nonumber \\
    &&\qquad+ \sum_{\{i,j\} \subseteq \Bar{A}} \det(K_{A})\det(K_{\{i,j\}} - K_{\{i,j\},A} K^{-1}_{A} K_{A,\{i,j\}}) +\nonumber \\
    &&\qquad \cdots + (-1)^{|\Bar{A|}} \det(K_{A}) \det(K_{\Bar{A}} - K_{\Bar{A},{A}} K^{-1}_{A} K_{A,\Bar{A}} ) \nonumber \\
    &=& (-1)^{|\Bar{A}|} \det(K_{A}) \det( (K_{\Bar{A}} - K_{\Bar{A},{A}} K^{-1}_{A} K_{A,\Bar{A}})-I_{\Bar{A}}) \nonumber \\
    &=&(-1)^{|\Bar{A}|} \det( K- I_{\Bar{A}})  ,\label{poly}
\end{eqnarray}
where $K_{A,B }$ denotes the matrix obtained from $K$ by   keeping only those  entries whose rows belong to $A$ and whose columns belong to $B$ (if $A = B$ we simply have  $K_{A}$.), $|A|$ denotes the cardinality of subset $A$, and $\Bar{A}$ the complement of set $A$. Here we use a slight abuse of notation of $I_{\bar{A}}$. We refer it to an N × N matrix whose restriction to $\bar{A}$ is $I_{\bar{A}}$ and has zeros everywhere else. Since $0 \preceq K\preceq I$, $\bbP (A = \bY) = |\det(K-I_{\Bar{A}})| \geq 0$.
This completes the proof of the theorem. 
\end{proof}
 
It is quite inconvenient to work with marginal kernels since their eigenvalues should be bounded by 0 and 1, and the marginal probability is not very appropriate to describe real world data. Here we introduce a slightly smaller class of DPPs called L-ensembles, where L only needs to be positive semi-definite.
\begin{definition}
	A point process is called an L-ensemble if it is defined through a real, symmetric, positive semi-definite $N\times N$ matrix $L$:
	\begin{equation}
		\bbP_{L}(A = \bY) \propto \det(L_{A}),
	\end{equation}
	where $A \subseteq \cY$ is a fixed subset.
\end{definition}
The proportional coefficient is given by \cite{KT12}:
$$\frac{1}{\sum_{A \subseteq \cY} \det(L_{A})} =\frac{1}{\det{(L +I)}}.$$
% \begin{theorem} [\cite{KT12}]\label{closedform}
%	For any $A \subseteq \cY$,
%	\begin{equation}
%		\sum_{A \subseteq Y \subseteq \cY} \det(L_{Y})
%		= \det(L + I_{\Bar{A}}).
%	\end{equation}
%	In particular, when $A = \emptyset$, we have   
%  $\sum_{A \subseteq \cY} \det(L_{A})= \det{(L +I) }$.
%  \end{theorem}
%\begin{proof}
%This can be proven by using the same argument as \ref{poly}.
%\end{proof}
%Thus we have
% \begin{equation}
%	\bbP_{L}(A = \bY) = \frac{\det(L_{A})}{\det(L+I)} \,.
% \end{equation}
Moreover, \cite{Mac75} shows that L-ensembles are indeed DPPs, where marginal kernels K is given by $ L(L+I)^{-1} = I - (L+I)^{-1}.$

\section{Maximum likelihood Estimator  of DPPs}

Let $ \mathcal{S}_{[N]}^{++}$  denote  the set of all $N$ by $N$ positive definite matrices.  Suppose that  $Z_{1},...,Z_{n} $ is a sample of size $n$      from a ${\text{DPP}(L^{\star})}$    for some given but  unknown $L^{\star} \in \mathcal{S}_{[N]}^{++}$. This  means that $Z_1, \cdots, Z_n$ are independent $\YY$ valued random subsets whose probability law is given by the ensemble $L^\star$.      The goal of this paper 
is to estimate this $L^\star$  from this given 
sample $Z_1, \cdots, Z_n$.   The most popular one is the maximum likelihood 
estimator $\hat N_n$. We shall recall its construction
this our specific setting and study its strong consistency, asymptotic normality, and associated Berry-Esseen theorem.

Let $P_{L}(Z_{i})$ denote the probability of the L-ensemble being $Z_{i}$.  The scaled log-likelihood function 
associated with the sample $Z_1, \cdots, Z_n$  is
\begin{equation}
\hat{\Phi}_n(L) = \frac{1}{n} \sum_{i=1}^{n} \log P_{L}(Z_{i})  = 
\sum_{J \subseteq [N]} \hat{p}{(J)}  \log \det(L_{J})- \log\det(L+I), \label{LF}
\end{equation}
where     $L \in \mathcal{S}_{[N]}^{++} $  and  $$\hat{p}{(J)}= \frac{1}{n} \sum_{i=1}^{n} \mathbb{I}(Z_{i} = J),$$
with $\mathbb{I}(\cdot)$ standing  for the indicator function.
It is also useful to define the expected log-maximum likelihood function given the real kernel $L^{\star}$
\begin{equation}
    \Phi_{L^{\star}}(L) = \sum_{J \subseteq [N]}{p_{L^\star}(J)}  \log \det(L_{J}) - \log\det(L+I)\,, 
\end{equation}
where 
\begin{equation*}
    p_{L^{\star}}(J) = \E{(\hat{p}{(J)})} = \frac{\det{(L^{\star}_{J})}}{\det{(L^{\star}+I)}}.
\end{equation*}
%Basically, we take the expectation of $\hat{p}{(J)}$ with respect to the true probability measure $\mathrm{DPP}(L^\star)$ and then get the expected maximum likelihood function. 
In the sequel   $L^{\star}$ is always fixed.  To simplify writing we  let $\hat{p}_{J}$ denote $\hat{p}(J)$, $p^*_{J}$ denote $p_{L^*}(J)$ and $\Phi$ denote $\Phi_{L^{\star}}$.
\begin{definition} A maximum likelihood estimator
	$\hat L_N$ is a global maximizer of \eqref{LF}:
	\begin{equation}
		\hat L_n=\argmax_{ 
			L\in \mathcal{S}_{[N]}^{++}} \hat \Phi_n(L) 
		 \,. \label{e.3.3} 
	\end{equation}  
\end{definition} 

% Let $\mathrm{KL}\big(\mathrm{DPP}(L^{\star}), \mathrm{DPP}(L)\big)$ be the Kullback-Leibler divergence, which measures the difference between distributions of $\mathrm{DPP}(L^{\star})$ and of $\mathrm{DPP}(L)$. Since Kullback-Leibler divergence is always non-negative, we have
% $$\mathrm{KL}\big(\mathrm{DPP}(L^{\star}), \mathrm{DPP}(L)\big) = \Phi{(L^{\star})} - \Phi(L) \geq 0, \, \forall L \in \SSN.  $$

We notice that $\Phi{(L^{\star})} - \Phi{(L)} $ is the Kullback-Leibler divergence $\mathrm{KL}\big(\mathrm{DPP}(L^{\star}), \mathrm{DPP}(L)\big)$, which measures the difference between distributions of $\mathrm{DPP}(L^{\star})$ and of $\mathrm{DPP}(L)$. KL divergence is always non-negative:
$$\mathrm{KL}\big(\mathrm{DPP}(L^{\star}), \mathrm{DPP}(L)\big) = \Phi{(L^{\star})} - \Phi(L) \geq 0, \, \forall L \in \SSN.  $$
As a consequence $L^{\star}$ is the global maxima of the expected maximum function $\Phi(L)$.

The next Lemma gives the differential of the likelihood function $\hat{\Phi}_n(L)$.
\begin{lemma} \label{mainlemma}
The gradient of log-likelihood function $\LF$ defined in \eqref{LF} exists and is given by
\begin{equation}
    \diff \LF = \sum_{J \subseteq [N]} \hp_{J} L_{J}^{-1} - (L+I)^{-1}. \label{e.3.4}
\end{equation}
\end{lemma}
\begin{proof}
 We regard determinant as a multivariate function of $N \times N$ variables and then the directional derivative of $\det(L+I)$ along direction $H$ is given by
\begin{align}
    \diff \det(L+I) (H) =& \lim_{t \to 0}  
    \frac{\det(L+I+tH) -\det(L+I)}{t}\nonumber \\
    =& \lim_{t \to 0} \det(L+I)\Big[\frac{\det(I+t(L+I)^{-1}H)-1}{t}\Big] \nonumber\\
    =& \lim_{t \to 0} \det(L+I) \Big[\frac{1+t\Tr((L+I)^{-1}H)+o(t^{2})-1}{t}\Big]\nonumber \\
    =& \det(L+I)\Tr((L+I)^{-1}H), \label{directional derivative}
\end{align}
where the third equality follows from the power series representation of $\det(I+A)$. Then the directional derivative of $\LF$ along direction $H$ is
\begin{equation}
    \diff \LF (H) = \sum_{J \subseteq [N]}\hp_{J} \Tr(L_{J}^{-1}H_{J}) - \Tr((L+I)^{-1}H).
\end{equation}
In matrix form, the above equation becomes
\begin{equation}
    \diff \LF = \sum_{J \subseteq [N]} \hp_{J} L_{J}^{-1} - (L+I)^{-1}.
\end{equation}
This proves the lemma
\end{proof}
\subsection{Strong consistency}
One critical issue  for the   maximum likelihood estimation is its  consistency. Due to non-identifiability of DPPs  illustrated  in Theorem 4.1 in \cite{kulesza2012learning} , any element in $\{DL^{\star}D: D \in \cD \}$ defines the same DPP, where $\cD$ is the collection of all diagonal matrices whose entry is either 1 or -1. Moreover, if $L_1, L_2\in  \mathcal{S}_{[N]}^{++}$  defines the same DPP,  then there is a $D\in \cD$ such that $L_2=DL_1D$. 
%Hence,  the global maxima is the set $\{DL^{\star}D: D \in \cD \}$.  
%Since determinantal point processes are not identifiable we measure the performance of 
Thus, we cannot expect that the maximum likelihood estimator $\hat  L_n$ will converge to $L^\star$. Instead, it should converge to $DL^\star D$ for some $D\in\cD$. 
%maximum likelihood estimation by the distance between the likelihood maximizer $\hat{L}_{n}$ and the set of true values:
Hence, for the consistency problem  we study the limit 
\begin{equation}
	\ell(\hat{L}_{n}, L^{\star}) := \min_{D\in \mathcal{D}} \Vert \hat{L}_{n} - DL^{\star}D \Vert_{F}.
	\label{e.3.8} 
\end{equation} 

\cite{brunel2017maximum} proves that this  distance converges to zero in probability. We shall prove a stronger version: the convergence  also holds almost surely. 
  
\begin{theorem} \label{strongconsistency}Let $Z_{1},...,Z_{n} $ be a  
	sample   of size $n$ from $Z \sim{\text{DPP}(L^{\star})}$.
	% for some unknown $L^{\star} \in E_{\alpha,\beta}$. 
	Let $\hat{L}_{n}$ be the maximum likelihood estimator of $L^\star$ defined by \eqref{e.3.3}.
	% defined on $E_{\alpha,\beta}$, then %$\ell(\hat{L}_{n},L^\star)$ converges to zero almost surely. 
	Then 	$ \ell(\hat{L}_{n}, L^\star)$ converges to zero almost surely.
\end{theorem}

\begin{proof} The proof is to combine   the convergence in probability result of \cite[Theorem 14]{brunel2017maximum} with  Wald's consistency theorem \cite{wald1949note} and is divided into the following 
	five  steps.  
%Even though the latter theorem originally requires the distribution to be identifiable, this is not a problem for this setting where we consider distance between $\hat{L}_{n}$ and the set of true values instead of one value.

\noindent {\bf Step 1}:   For $0<\alpha < \beta< 1$, define a set $E_{\alpha,\beta}$
$$E_{\alpha,\beta} = \big\{ L \in \mathcal{S}^{++}_{[N]} : K=L(I+L)^{-1} \in \mathcal{S}^{[\alpha,\beta]}_{[N]} \big\}, $$
where $\mathcal{S}^{[\alpha,\beta]}_{[N]}$ is the set of all N by N symmetric matrices whose eigenvalues are between $\alpha$ and $\beta$. Observe that $E_{\alpha,\beta}$ is compact since it's bounded and closed in $\mathbb{R}^{N \times N}$.  We first show that $\ell(\hat{L}_{n}, L^\star)$ converges to zero almost surely when parameters of matrices are restricted  to  $L^\star \in E_{\alpha,\beta}$ 
 for  appropriate $0<\alpha< \beta <1$.    
%\begin{lemma} \label{LemmaW}
%Let $Z_{1},...,Z_{n} $ be $n$ independent subsets of $Z \sim{\text{DPP}(L^{\star})}$ for some unknown $L^{\star} \in E_{\alpha,\beta}$. Let $\hat{L}_{n}$ be the maximum likelihood estimator of $L^{\star}$ defined on $E_{\alpha,\beta}$, then $\ell(\hat{L}_{n},L^\star)$ converges to zero almost surely.
%\end{lemma}
% 
In order to complete this proof, we let 
$$ 
\Delta \hat{\Phi}(L) = \hat{\Phi}(L) - \hat{\Phi}(L^{\star}) = \frac{1}{n}\sum_{i=1}^{n} \log\frac{P_{L}(Z_i)}{P_{L^\star}(Z_i)}$$
and
$$
\Delta {\Phi}(L) = {\Phi}(L) - {\Phi}(L^{\star})
=\E_{L^\star} \big(\log\frac{P_{L}(Z)}{P_{L^\star}(Z)} \big).$$

$ \Delta {\Phi}(L)$ is the Kullback-Leibler divergence between DPP($L^\star$) and DPP($L$). By Jensen's inequality, $ \Delta {\Phi}(L) \leq 0$ for all $L$ and  $ {\Phi}(L) =  {\Phi}(L^{\star}) $ if and only if $P_{L}(Z) = P_{L^\star}(Z) $ for all $Z\in[N]$, which means $L = DL^\star D$ for some $D \in \mathcal{D}$. In the sequel let $E$ denote $E_{L^*}$

For each $L \in E_{\alpha,\beta} $, the strong law of large numbers implies
\begin{equation}
	\Delta {\hat{\Phi}}(L) \xrightarrow[]{a.s.}\Delta {\Phi}(L). \label{e.3.9} 
\end{equation} 

\noindent {\bf Step 2}: However, the above convergence
\eqref{e.3.9}  doesn't imply the convergence of maximum likelihood estimator to the true values. Thus  the Wald's integrability condition is needed:
for every $ L \in E_{\alpha, \beta}$, there exists $\epsilon > 0$ such that,
\begin{equation}
    \E \sup_{\substack{N\in E_{\alpha, \beta}\\ \ell(L,N) < \epsilon}} \log \frac{P_{N}(Z)}{P_{L^{\star}}(Z)} < \infty \label{W}.
\end{equation}
Since $L \mapsto \log \frac{P_{L}(Z)}{P_{L^{\star}}(Z)} $ is continuous (the determinant function is continuous),  for any   $\delta >0$ there exists $\epsilon >0$, when $\ell(L,N)
 < \epsilon$ $$(1-\delta)\frac{P_{L}(Z)}{P_{L^{\star}}(Z)}< \frac{P_{N}(Z)}{P_{L^{\star}}(Z)} <(1+\delta) \frac{P_{L}(Z)}{P_{L^{\star}}(Z)}.$$
Then the Wald's integrability condition is satisfied.
Now for every sequence $\{L_{n}\}$ converging to L, we show that $ \Delta\Phi(L_{n})$ is upper semicontinuous:
\begin{eqnarray}
    \limsup_{n\to \infty} \Delta\Phi(L_{n}) &=& \limsup_{n\to \infty} \E \log\frac{P_{L_{n}}(Z)}{P_{L^\star}(Z)} \nonumber \\
    &\leq& \E\limsup_{n\to \infty} \log \frac{P_{L_{n}}(Z)}{P_{L^\star}(Z)} \nonumber \\
    &=&\E \log \frac{P_{L}(Z)}{P_{L^\star}(Z)} \nonumber \\
    &=& \Delta \Phi(L) \nonumber.
\end{eqnarray} 
The second inequality follows from the Fatou's lemma and the third identity is the consequence of continuity of the function $\log\frac{P_{L_{n}}(Z)}{P_{L^\star}(Z)}$.
For every $\eta > 0$ we define the set $K_{\eta}$
\begin{align}
     K_{\eta} &= \big\{ L \in E_{\alpha,\beta}: \ell(L,L^\star) \geq \eta \big\} \nonumber \\
    ~&= \bigcap_{D\in\mathcal{D}} \big\{L \in E_{\alpha,\beta}: \Vert L - DL^{\star}D \Vert_{F} \geq \eta \big\},
\end{align}
which is  closed   and hence   compact. 

Since $\Delta \Phi(L)$ is an upper semicontinuous function, it achieves maximum over the compact set $K_{\eta}$
(since $E_{\alpha, \beta}$ is compact). We denote the maximum by $m(\eta)$. And we cannot have $m(\eta) = 0$ because that would imply there is an  $L \in K_{\eta}$ such that $L = DL^\star D$ for some $D \in \mathcal{D}$.
The strong law of large numbers implies
\begin{eqnarray}
    \sup_{\substack{N\in E_{\alpha, \beta}\\ \ell(L,N) < \epsilon}} \Delta \hat{\Phi}(N) &\leq& \frac{1}{n} \sum_{i=1}^{n} \sup_{\substack{N\in E_{\alpha, \beta}\\ \ell(L,N) < \epsilon}} \log \frac{P_{N}(Z_{i})}{P_{L^\star}(Z_{i})} \nonumber \\
    &\xrightarrow[]{a.s.}& \E \sup_{\substack{N\in E_{\alpha, \beta}\\ \ell(L,N) < \epsilon}} \log \frac{P_{N}(Z)}{P_{L^\star}(Z)}.
\end{eqnarray}

By continuity,
$$\lim_{\epsilon \to 0}\sup_{\substack{N\in E_{\alpha, \beta}\\ \ell(L,N) < \epsilon}} \log \frac{P_{N}(Z))}{P_{L^\star}(Z)} = \log\frac{P_{L}(Z)}{P_{L^\star}(Z)} $$  
and $\sup_{\epsilon} \log\frac{P_{N}}{P_{L^\star}}$
is a decreasing function with respect to $\epsilon$ because supremum over a smaller subset is smaller than that  over a bigger subset. And by \eqref{W} it is integrable for all small enough $\epsilon$. Hence by  the dominated convergence theorem,
\begin{equation}
    \lim_{\epsilon \to 0} \E\sup_{\substack{N\in E_{\alpha, \beta}\\ \ell(L,N) < \epsilon}} \log \frac{P_{N}(Z))}{P_{L^\star}(Z)} = \E \log\frac{P_{L}(Z)}{P_{L^\star}(Z)} = \Delta \Phi(L) \nonumber.
\end{equation}
Thus for any $L \in K_{\eta}$ and any $\gamma > 0$ there exists an \ $\epsilon_{L}$ such that 
\begin{equation}
    \E\sup_{\substack{N\in E_{\alpha, \beta}\\ \ell(L,N) < \epsilon_{L}}} \log \frac{P_{N}(Z)}{P_{L^\star}(Z)} < m(\eta)+\gamma \label{gamma} .
\end{equation}

\noindent{\bf Step 3}.  For each $L \in K_{\eta}$, we define the open set:
$$V_{L} = \{N\in E_{\alpha, \beta}: \ell(N,L) < \epsilon_{L}\}$$
and then the family $\{V_{L}: L \in K_{\eta}\}$ is an open cover of $K_{\eta}$ and hence has a finite subcover: $V_{L_{1}}, V_{L_{2}}, .... , V_{L_{d}}$.
On every $V_{L_{i}}$ we use strong law of large numbers again to obtain 
\begin{eqnarray}
    \limsup_{n \to \infty} \sup_{N\in V_{L_{i}}} \Delta\hat{\Phi}(N) &\leq& \limsup_{n \to \infty} \frac{1}{n}\sum_{i=1}^{n}\sup_{N\in V_{L_{i}}}\log\frac{P_{N}(Z_{i})}{P_{L^\star}(Z_{i})} \nonumber \\
    &=& \E \sup_{N\in V_{L_{i}}}\log\frac{P_{N}(Z)}{P_{L^\star}(Z)}.
\end{eqnarray}
From \eqref{gamma} we get
$$    \limsup_{n \to \infty} \sup_{N\in V_{L_{i}}} \Delta\hat{\Phi}(N) < m(\eta) +\gamma \qquad i = 1,2,...,d.$$
Since $\{V_{L_{i}}: i=1,2 ...,d \}$ cover $K_{\eta}$ we have
$$    \limsup_{n \to \infty} \sup_{N\in K_{\eta} } \Delta\hat{\Phi}(N) < m(\eta) +\gamma $$
which, since $\gamma$ is arbitrary, implies
\begin{equation}
\limsup_{n \to \infty} \sup_{L \in K_{\eta} } \Delta \hat{\Phi}(L) < \sup_{L \in K_{\eta}} \Delta {\Phi}(L) = m(\eta) \label{sup}.
\end{equation}
Notice that $m(\eta) < 0$. From \eqref{sup} there exists a constant $N_{1}$ such that
$$\sup_{L \in K_{\eta} } \Delta \hat{\Phi}(L) < \frac{m(\eta)}{2}, \qquad n > N_{1}.$$
But $$ \Delta \hat{\Phi}(\hat{L}_{n}) = \sup_{L \in E_{\alpha,\beta}}  \Delta \hat{\Phi}({L}) \geq \Delta \hat{\Phi}(L^\star) \xrightarrow[]{a.s.} \Delta{\Phi}(L^\star) =0 , $$
so there exists a constant $N_{2}$ such that
$$\Delta \hat{\Phi}(\hat{L}_{n}) \geq \frac{m(\eta)}{2}, \qquad n>N_{2}$$
which implies that $\hat{L}_{n} \notin K_{\eta}$, that is,
$\ell(\hat{L}_{n}, L ) < \epsilon $.

\noindent{\bf Step 4}. Now we can remove the compactness condition. 
 First  we  show that the event $\{\hat{L}_{n} \in E_{\alpha,\beta}\}$ holds almost surely. We adopt the proof from \cite{brunel2017maximum}.
	Let $\delta = \min_{J\subset [N]}P_{L^\star}{(J)}$. For simplicity, we denote $P_{L^\star}{(J)}$ by $p^{\star}_{J}$.  Since $L^\star$ is positive definite, $\delta >0$. Define the event $\mathcal{A}_{n}$ by
	
	$$\mathcal{A}_{n} = \bigcap_{J \subset[N]} \big\{ p^{\star}_{J} \leq 2\hat{p}_{J}\leq 3p^{\star}_{J}\big\}.$$
	Observe that $\Phi(L^\star)<0$ and  we can find $\alpha < \exp(3\Phi(L^\star)/\delta)$ and $\beta >1- \exp(3\Phi(L^\star)/\delta)$ such that $0<\alpha<\beta<1$.
	Then using   \cite[Theorem 14]{brunel2017maximum} we know that on the event $\mathcal{A}_{n}$, $\hat{L}_{n} \in E_{\alpha, \beta}$, that is,
	$$P(\hat{L}_{n} \in E_{\alpha, \beta}) \geq P(\mathcal{A}_{n}).$$
	Because $$\hat{p}_{J} = \frac{1}{n}\sum_{i=1}^{n}\mathbb{I}(Z_{i}=J) \xrightarrow[]{a.s.}P_{L^\star}(Z=J) = p^{\star}_{J}\,, $$ 
	the event $\mathcal{A}_{n}$ holds almost surely when $n$ goes to infinity and hence $\{\hat{L}_{n} \in E_{\alpha,\beta}\}$ holds almost surely. 
	
\noindent{\bf Step 5}.  Let $\mathbb{I}_{E_{n}}$ denote the indicator function of the event $ \{ \hat{L}_{n} \in E_{\alpha, \beta} \} $.  Then
\begin{eqnarray}
    \bbP\big(\lim_{n \to \infty} \ell(\hat{L}_{n},L^{\star}) = 0 \big)&=& \bbP\big(\lim_{n \to \infty} \ell(\hat{L}_{n} ,L^{\star}) = 0 , \lim_{n \to \infty} \mathbb{I}_{E_{n}} =1 \big) \nonumber \\
    &\,& +\bbP\big(\lim_{n \to \infty} \ell(\hat{L}_{n} ,L^{\star}) = 0 , \lim_{n \to \infty} \mathbb{I}_{E_{n}} \neq 1\big )\nonumber \\
    &=& \bbP \big(\lim_{n \to \infty} \ell(\hat{L}_{n} ,L^{\star}) = 0 , \lim_{n \to \infty} \mathbb{I}_{E_{n}} =1 \big)   \nonumber\\
    &=& \bbP\big(\lim_{n \to \infty} \ell(\hat{L}_{n} ,L^{\star}) = 0  \big| \lim_{n \to \infty} \mathbb{I}_{E_{n}} =1 \big)\bbP\big( \lim_{n \to \infty} \mathbb{I}_{E_{n}} = 1\big)  \nonumber \\
    &=&\bbP\big(\lim_{n \to \infty} \ell(\hat{L}_{n} ,L^{\star}) = 0  \big| \lim_{n \to \infty} \mathbb{I}_{E_{n}} =1 \big)\nonumber \\
    &=& 1 \nonumber.
\end{eqnarray}
The last equality follows from the fact that $\hat{L}_{n} \in E_{\alpha,\beta} $ almost surely proved in Step 4.
\end{proof}

\subsection{Berry-Esseen theorem}
An $N$ by $N$ matrix $[A_{ij}]_{N\times N}$ can also be viewed as an $N \times N$ dimensional column vector: $ (A_{11}, A_{12}, ..., A_{1N}, A_{21}, ..., A_{N1}, ... A_{NN})^{T}$. With this identification the  Frobenius norm of the matrix $||\cdot||_{F}$ is just the $\mathcal{L}^{2}$ norm for its corresponding column vector. In the following   we  shall regard the matrix as the corresponding column vector.

 Because of non-identifiability of DPPs,  the maximum likelihood estimators are not unique. We choose the estimator $\Tilde{L}_n$ which is closest to the  given   true value $L^*$. In fact, let ${\hat{L}_n}$ be one maximal likelihood estimator. Let $\hat{D}_n \in \mathcal{D} $ be such that  
\begin{equation} \label{theMLE}
    \lVert\hat{D}_n\hat{L}_n\hat{D}_n - L^{\star} \rVert_{F} = \min_{{{D} \in \mathcal{D}}} 
    \lVert {D} \hat{L}_n{D}   - L^{\star} \rVert_F
\end{equation}
and set 
\begin{equation}
	\Tilde{L}_n= \hat{D}_n\hat{L}_n\hat{D}_n\,. \label{e.3.16a}
\end{equation}  Then the strong consistency of $\Tilde{L}_n$ immediately follows from   Theorem \ref{strongconsistency}.

We impose irreducibility assumption introduced in \cite{brunel2017maximum} on $L^{\star}$, of which the assumption is equivalent to $L^{\star}$ being a one block matrix, that is, it cannot be partitioned into two block matrices. And then according to   \cite[Theorem 8]{brunel2017maximum}, $\mathrm{d}^{2} {\Phi}(L^{\star})$ is negative definite and hence invertible. Let $V(L^{\star})$ denote its inverse:
\begin{equation}
V(L^{\star}):=\left[\mathrm{d}^{2} {\Phi}(L^{\star})\right]^{-1}\,. \label{e.3.17}
\end{equation}
Here if we vectorize ${L}$ then $\mathrm{d}^{2} {\Phi}(L^{\star})$ is an $(N \times N) \times (N \times N)$ Hessian matrix. By \cite[Theorem 5.41]{van2000asymptotic}, 
\begin{eqnarray}
    \sqrt{n}(\Tilde{L}_n   - L^{\star})
    &=& -({\E{(\mathrm{d}^{2}\log {P_{L^{\star}}(Z)})}})^{-1} \frac{1}{\sqrt{n}} \sum_{i=1}^{n} \mathrm{d}( \log P_{L^{\star}}(Z_{i})) + o_{P}(1)\nonumber \\
    ~&=& -V(L^{\star})\frac{1}{\sqrt{n}} \sum_{i=1}^{n}((L^{\star}_{Z_{i}})^{-1} - (I +L^{\star})^{-1}) + o_{P}(1).\label{5.41} 
\end{eqnarray}
%where $\Vert \rho_{n}\Vert_{F} \xrightarrow[n \longrightarrow \infty]{} 0 $.  
In particular, \cite[Theorem 5.41]{van2000asymptotic}  states that the sequence $ \sqrt{n}(\Tilde{L}_n  - L^{\star})$ is asymptotically normal with mean $\bm{0}$ and covariance matrix $ -V(L^{\star})$. Hence we get the following theorem for the asymptotic normality of the maximum estimator $\Tilde{L}_n$.
\begin{theorem} \label{CLT}
Let $L^{\star}$ be irreducible. Then, $\Tilde{L}_n$ is asymptotically normal:
\begin{equation}
    \sqrt{n}(\Tilde{L}_n - L^{\star})  \stackrel{d}{\rightarrow}   \mathcal{N} (\bm{0},-V(L^{\star})), \quad \hbox{  as \ $n \longrightarrow \infty$} \label{e.3.15} 
\end{equation}
where $V(L^{\star})$ is given by \eqref{e.3.17} and   the above convergence $\stackrel{d}{\rightarrow}$  holds in distribution.
\end{theorem}

Having established the  asymptotic normality of the maximum estimator $\Tilde{L}_n$,    we want to  take one step further. 
%Namely, we  would like  to find the rate of convergence of \eqref{e.3.15}.
Namely, 
we want to find  an upper error bound on the rate of convergence of the distribution of  $ (-V(L^{\star}))^{-\frac{1}{2}} \sqrt{n}(\Tilde{L}_n - L^{\star}) $ to standard multidimensional normal distribution $Z \sim \mathcal{N}(\bm{0},{I})$. In other words, we are going to obtain a Berry-Esseen type theorem.   We argue that when $\TL \in E_{\alpha, \beta} $, the bound of the maximal error is of order $n^{-\frac{1}{4}}$. This restriction  is not a big deal. Indeed, since $\alpha$ and $\beta$ can be chosen arbitrarily   close to $0$ and $1$ respectively so that  $E_{\alpha,\beta}$ converges to $\SSN$. What's more, since from Theorem \ref{strongconsistency}, $\hat{L} \in E_{\alpha,\beta}$ almost surely,  $ \hat{D}\hat{L}\hat{D} = \TL \in E_{\alpha,\beta}$ almost surely. %so there is always a subsequence of $\TL$ satisfying this condition. 
\begin{theorem} \label{BETT}
    Let the maximum likelihood estimator $\Tilde{L}_n$ be   defined by \eqref{e.3.16a}  and also belong to $E_{\alpha, \beta}$ for some $0<\alpha<\beta<1$ and let Z be an $N \times N$ standard Gaussian matrix. Then for every $x \in \mathbb{R}^{N \times N}$,
    $$ \vert \bbP((-V(L^{\star}))^{-\frac{1}{2}}\sqrt{n} (\TL - L^{\star}) < x) - \bbP(Z <x ) \vert \leq C \frac{1}{\sqrt[4]{n}},
    $$
where C is a sufficiently  large constant, which is irrelevant to $x$, subject to $\alpha, \beta$ and proportional to $N^2$. Here for two 
$N$ by $N$ matrices $A$ and $B$, $A< B$  means the   components of $A$ is less than the corresponding components of $B$. 
\end{theorem}
\begin{proof} We divide the proof into four steps.
	
\noindent{\bf Step 1}. \ 	
According to \eqref{5.41}, $ (-V(L^{\star}))^{-\frac{1}{2}} \sqrt{n}(\Tilde{L}_n - L^{\star}) $ can be decomposed into a sum 
\begin{equation} \label{decompose}
X_{n} =\sum_{i=1}^{n} \xi_{i} := (-V(L^{\star}))^{\frac{1}{2}}\frac{1}{\sqrt{n}} \sum_{i=1}^{n}((L^{\star}_{Z_{i}})^{-1} - (I +L^{\star})^{-1}) 
\end{equation}
and a term $\rho_{n}=(-V(L^{\star}))^{-\frac{1}{2}}o_{P}(1) $ whose Frobenius norm converges to zero in probability.  
\begin{align}
    &\vert \bbP(X_{n} + \rho_{n} < {x})  - \bbP(Z < {x}) \vert \nonumber \\
    &= \vert \bbP(X_{n} + \rho_{n} < {x}, \Vert\rho_{n} \Vert_{F} \geq k_{n})+\bbP(X_{n} + \rho_{n} < {x}, \Vert\rho_{n} \Vert_{F} < k_{n}) 
    - \bbP(Z < {x}) \vert \nonumber \\
    &\leq \bbP(\Vert \rho_{n}\Vert_{F} \geq k_{n}) 
    + \vert \bbP(X_{n} + \rho_{n} < {x}, \Vert\rho_{n} \Vert_{F} < k_{n}) - \bbP(Z < {x})\vert \nonumber \\
    &\leq \bbP(\Vert \rho_{n}\Vert_{F} \geq k_{n}) \nonumber\\ %\tag{I1} \label{I1}\\
    &\qquad+  \vert \bbP(X_{n} + {k_{n}}{\mathbbm{1}} < x, \Vert\rho_{n} \Vert_{F} < k_{n}) - \bbP(Z < {x})\vert \nonumber \\ % \tag{I_2} \label{I_2}  \\
    &\qquad+ \vert \bbP(X_{n} - {k_{n}}{\mathbbm{1}} < x, \Vert\rho_{n} \Vert_{F} < k_{n}) - \bbP(Z < {x})\vert \nonumber%\tag{I_3}  \label{I_3},
   \nonumber \\
    &=: I_1+I_2+I_3\,, 
\end{align}
where $\{k_{n}\}$ is an arbitrary sequence of positive real number and $\mathbbm{1}$ is the $N \times N$ matrix whose entries are all 1.  

\noindent{\bf Step 2}. The Estimation of  ($I_1$).  We claim 
  $
    \bbP(\Vert \rho_{n}\Vert \geq k_{n}) \leq  \frac{C_{4}}{\sqrt[4]{n}},
    $
where $k_{n} = n^{-\frac{1}{4}}$ and $C_{4}$ is a constant. 

In fact, 
from the proof of   \cite[Theorem 5.41]{van2000asymptotic}, $\rho_{n}$ has the following expression
\begin{align}
    \rho_{n} = \sqrt{n} (-V(L^{\star}))^{\frac{1}{2}}\bigg( &\di^2 \hat{\Phi}_{n}(L^{\star}) - \E(\di^2 \hat{\Phi}_{n}(L^{\star})) \notag \\
    &+ \frac{1}{2}(\Tilde{L}_n-L^{\star})^{T} \di^{3}\hat{\Phi}_{n}(L_{n})\bigg)(\Tilde{L}_n-L^{\star}),
\end{align}
where $L_{n}$ is a point on the line segment between $\Tilde{L}_n$ and $L^{\star}$. To simplify notation, let $\theta$ denote  $$\bigg( \di^2 \hat{\Phi}_{n}(L^{\star}) - \E(\di^2 \hat{\Phi}_{n}(L^{\star})) + \frac{1}{2}(\Tilde{L}_n-L^{\star})^{T} \di^{3}\hat{\Phi}_{n}(L_{n})\bigg)(\Tilde{L}_n-L^{\star}).$$
Then
\begin{eqnarray}
    \E{\Vert \rho_{n} \Vert_{F}} &=& \E\Vert\sqrt{n}(-V(L^{\star}))^{\frac{1}{2}}\theta\Vert_{F}\nonumber \\
    &=& \sqrt{n} \E\Vert(-V(L^{\star}))^{\frac{1}{2}}\theta\Vert_{F} \nonumber \\
    &\leq& \sqrt{n} \E\Vert (-V(L^{\star}))^{\frac{1}{2}}\Vert_{op}\Vert\theta \Vert_{2} \nonumber \\
    &=& \sqrt{n \cdot \Lambda_{max}(-V)} \cdot  \E \Vert \theta \Vert_{2}.
\end{eqnarray}
$\Vert \cdot \Vert_{op}$ denotes the operator norm induced by $\mathcal{L}^{2} $ norm and $\Lambda_{max}$ denotes the largest eigenvalue. For the first inequality, we regard $\theta$ as an $N \times N$ column vector and $(-V(L^{\star}))^{\frac{1}{2}}$ is an $(N \times N) \times (N \times N)$ matrix.
\begin{align}
    \E \Vert \phi \Vert_{2} 
    =& \E \big\Vert\big( \di^2 \hat{\Phi}_{n}(L^{\star}) - \E(\di^2 \hat{\Phi}_{n}(L^{\star})) + \frac{1}{2}(\Tilde{L}_n-L^{\star})^{T} \di^{3}\hat{\Phi}_{n}(L_{n})\big)(\Tilde{L}_n-L^{\star})\big\Vert_{2}\nonumber  \\
    \leq& \E \big\Vert\big( \di^2 \hat{\Phi}_{n}(L^{\star}) - \E(\di^2 \hat{\Phi}_{n}(L^{\star}))\big)(\Tilde{L}_n-L^{\star})\big\Vert_{2}  %\tag{I_1-1}
    \label{rho1}  \\ 
    +& \E \Vert \frac{1}{2}(\Tilde{L}_n-L^{\star})^{T} \di^{3}\hat{\Phi}_{n}(L_{n})(\Tilde{L}_n-L^{\star})\big\Vert_{2}\,.  %\tag{I_1-2} 
    \label{rho2}
\end{align}
Using Cauchy-Schwartz inequality to estimate \eqref{rho1} we see 
\begin{eqnarray} 
     I_1-1  &\leq& \E^{\frac{1}{2}} \big\Vert \di^2 \hat{\Phi}_{n}(L^{\star}) - \E(\di^2 \hat{\Phi}_{n}(L^{\star}))\big\Vert^{2}_{op} \E^{\frac{1}{2}} \Vert \Tilde{L}_n-L^{\star}\Vert^{2}_{2} \nonumber \\
    &\leq& \frac{N^{2}}{\sqrt{n}}\max_{i,j}({L^{\star}}^{-1})_{ij}^{2} \E^{\frac{1}{2}} \Vert \Tilde{L}_n-L^{\star}\Vert^{2}_{2}.
\end{eqnarray} 
Let $h(x)$ be a multivariate function:
\begin{align*}
  h: \quad  \mathbb{R}^{N \times N} \longrightarrow{} & \, \mathbb{R} \\
    (x_{1}, x_{2},..., x_{NN}) \longmapsto& \, x_{1}^{2} + x_{2}^{2} + \cdots + x_{NN}^{2}
\end{align*}
Then $h$ is a continuous function. What's more almost surely $\TL \in E_{\alpha,\beta}$, which is a compact and convex set. Using Theorem \ref{CLT} and portmanteau lemma we have
\begin{equation}
    \E \big(h(\sqrt{n}(\TL -L^{\star}))\big)= n \E\Vert \TL - L^{\star} \Vert^{2}_{F} \longrightarrow \E \Vert \Tilde{Z} \Vert^{2}_{F} ,
\end{equation}
where $\Tilde{Z} \sim \mathcal{N}(\bm{0}, -V(L^\star))$. $\E \Vert \Tilde{Z} \Vert^{2}_{F}$ is equal to $\E(\Tilde{Z}_{11}^2 + \cdots + \Tilde{Z}_{1n}^2 + \Tilde{Z}_{21}^2 + \cdots + \Tilde{Z}_{nn}^2) = \mathrm{Tr}(-V(L^{\star}))$. Then there exists a constant $C_{1}$ subject to $\alpha, \beta$ such that
\begin{equation}
    \E^{\frac{1}{2}} \Vert \Tilde{L}_n-L^{\star}\Vert^{2}_{2} \leq C_{1}\frac{1}{\sqrt{n}}.
\end{equation}
As a result,
\begin{equation}
    \text{\eqref{rho1}} \leq C_{2} N^{2} \frac{1}{n}\,, 
\end{equation}
where $C_{2}$ is a suitable constant.

Next, we estimate the second part, that is \eqref{rho2}:
$$ \E \Vert \frac{1}{2}(\Tilde{L}_n-L^{\star})^{T} \di^{3}\hat{\Phi}_{n}(L_{n})\big)(\Tilde{L}_n-L^{\star})\big\Vert_{2}.$$
Here $\di^{3}\hat{\Phi}_{n}(L_{n})$ is an $N \times N$ dimensional column vector whose entries are $N \times N$ matrices. Since $\LF$ is infinitely many time differentiable, $L_{n}$ is on the line segment between $\Tilde{L}_n$ and $L^{\star}$, and $E_{\alpha,\beta}$ is a convex and compact set, we conclude that every entry of $\di^{3}\hat{\Phi}_{n}(L_{n})$ is bounded. Hence there exists a constant $C_{3} \geq 0$ subject to $\alpha$ and $\beta$ such that
\begin{align}
    \E \Vert \frac{1}{2}(\Tilde{L}_n-L^{\star})^{T} \di^{3}\hat{\Phi}_{n}(L_{n})\big)(\Tilde{L}_n-L^{\star})\big\Vert_{2} \leq& C_{3} \E \Vert \Tilde{L}_n-L^{\star}\Vert^{2}_{2} \nonumber \\
     \leq& \frac{C_{1}^{2}C_{3}}{n}.
\end{align}
Now let $k_{n} = n^{-\frac{1}{4}}$. Using Chebyshev's inequality we get:
\begin{equation} 
    \bbP(\Vert \rho_{n}\Vert_{F} \geq k_{n}) \leq \frac{\E \Vert \rho_{n} \Vert_{F}}{k_{n}} = \frac{C_{4}}{\sqrt[4]{n}} \label{e.3.32}
\end{equation}
for a suitable constant $C_{4}$.

\noindent{\bf Step 3}. 
Our next goal is to estimate  ($I_2$) as follows. 
Let $k_{n}$ be $\frac{1}{\sqrt[4]{n}}$. Then 
  \[
   I_2\leq \frac{C_{7}}{\sqrt[4]{n}}\]
   for some constant $C_{7}$.

Because
\begin{align}
&\bbP(X_{n}+k_{n}\mathbbm{1} <x) -\bbP(Z<x) \nonumber \\
&\geq \bbP(X_{n} + {k_{n}}\mathbbm{1} < x, \Vert\rho_{n} \Vert_{F} < k_{n}) - \bbP(Z < {x}) \nonumber \\
    &=  \big(\bbP(X_{n}+k_{n}\mathbbm{1} < x)-\bbP(X_{n} + {k_{n}}\mathbbm{1} < x, \Vert\rho_{n} \Vert_{F} \geq k_{n})\big) - \bbP(Z < {x}) \nonumber \\
    &\geq \bbP(X_{n}+k_{n}\mathbbm{1} < x)- \bbP(\Vert\rho_{n} \Vert_{F} > k_{n}) - \bbP(Z < {x}) \nonumber,
\end{align}
we have
%\begin{spacing}
\begin{align}
  I_2
    &\leq \vert \bbP(X_{n}+k_{n}\mathbbm{1} <x) -\bbP(Z<x)\vert \nonumber \\
    & \quad + \vert \bbP(X_{n}+k_{n}\mathbbm{1} < x)- \bbP(\vert \rho_{n} \Vert_{F} \geq k_{n}) - \bbP(Z < {x})\vert \nonumber \\
    &\leq 2 \vert \bbP(X_{n}+k_{n}\mathbbm{1} <x) -\bbP(Z<x)\vert + \bbP(\Vert \rho_{n} \Vert_{F} \geq k_{n}) \nonumber \\
    &= 2 \vert \bbP(X_{n}+k_{n}\mathbbm{1} <x) -\bbP(Z +k_{n}\mathbbm{1} <x) \nonumber \\
    & \quad + \bbP(Z +k_{n}\mathbbm{1} <x) - \bbP(Z<x)\vert + \bbP(\Vert \rho_{n} \Vert_{F} \geq k_{n}) \nonumber \\
    &\leq 2 \vert \bbP(X_{n}+k_{n}\mathbbm{1} <x) -P(Z +k_{n}\mathbbm{1} <x) \vert %\tag{I_2-1}
     \nonumber\\ %\label{BET}\\
    &\quad +  2\vert \bbP(Z +k_{n}\mathbbm{1} <x)- P(Z<x) \vert %\tag{I_2-2}
%    \nonumber\\ %\label{Guassian}\\
%    &\quad 
    +  \bbP(\Vert \rho_{n} \Vert_{F} \geq k_{n}) \nonumber \\ %\tag{I_2-3}\\
   &=I_{21}+I_{22}+I_{23} \label{Chebyshev}.
\end{align}
%\end{spacing}
By multidimensional Berry-Esseen theorem in \cite{bentkus2005lyapunov}, 
\begin{equation}
    I_{21} \leq C_{5} \cdot \sqrt{N} \cdot n \cdot \E{\Vert \xi_{1} \Vert_{2}^{3}}
\end{equation}
where $C_{5}$ is a constant and $\xi_{1}$ is defined in \eqref{decompose}:
\begin{eqnarray}
    \E{\Vert \xi_{1} \Vert}^{3} &=& \E\Vert\frac{1}{\sqrt{n}}(-V(L^{\star}))^{-\frac{1}{2}}\big((L^{\star}_{Z_{i}})^{-1} - (I +L^{\star})^{-1} \big)\Vert_{2}^{3} \nonumber \\
    &\leq& (\frac{1}{\sqrt{n}})^{3}\E\Vert(-V(L^{\star}))^{-\frac{1}{2}}\big((L^{\star}_{Z_{i}})^{-1} - (I +L^{\star})^{-1}\big) \Vert^{3}_{2}.
\end{eqnarray}
Since $\E\Vert(-V(L^{\star}))^{-\frac{1}{2}}\big((L^{\star}_{Z_{i}})^{-1} - (I +L^{\star})^{-1}\big) \Vert_{2}^{3}$ is a constant we get
\begin{equation}
    I_{21} \leq C_{6}\sqrt{\frac{N}{n}} \label{BETR}
\end{equation}
For $I_{22}$, since Z can be viewed as a standard Gaussian random vector, we have
\begin{eqnarray}
    I_{22} &=& 2\vert \bbP(x-k_{n}I <Z_{n} <x) \vert \nonumber \\
    &\leq& 2\sum_{i,j = 1}^{N} \bbP(x_{ij}-k_{n} \leq (Z_{n})_{ij}\leq x_{ij}) \nonumber \\
    &=& \frac{2N^{2}}{\sqrt{2\pi}}k_{n} \label{GuassianR}
\end{eqnarray}
Combining \eqref{BETR}, \eqref{GuassianR} with previous bound \eqref{e.3.32} to treat $I_{23}$, where we take $k_{n} = n^{-\frac{1}{4}}$ we conclude that
$$I_2 \leq \frac{C_{7}}{\sqrt[4]{n}}, $$
where $C_{5}$ is a constant.

\noindent{\bf Step 4}. 
As for $I_3$ we can use the same argument as above and conclude  that  $I_3$ is 
bounded by  $C_{8} \cdot n^{-\frac{1}{4}}$ for some constant $C_{8}$.
This completes the proof of the theorem. 
\end{proof}
%\begin{proof}{[Theorem \ref{BETT}]}
%\\
%The result follows from the last two lemmas.
%\end{proof}

\section{Simulation study}
In this section we test numerically the mle for the DPPs. Since in general
there is no analytic expression for the mle \eqref{theMLE}
we need to use numerical algorithms  to find the mle. We try  the well-known  Newton-Raphson algorithm and compare it with stochastic gradient descent algorithm.
% on learning the  DPP kernels. 

%
%\subsection{Newton-Raphson method}
%
%In this subsection we develop Newton-Raphson algorithm for finding the maximum likelihood estimator of determinantal  point processes in general case. 

\medskip 
\noindent{\bf Algorithms}.   First, we explain  the formula of the Newton-Raphson  algorithm applied to our mle of DPPs. Recall that the gradient of the maximum likelihood function is  given by \eqref{e.3.4}: 
\begin{equation}
    \diff \LF = \sum_{J \subseteq [N]} \hp_{J} L_{J}^{-1} - (L+I)^{-1}. \nonumber 
\end{equation}

To use Newton-Raphson algorithm  we think of $\diff \LF $ as an $N \times N$ vector-valued function $ (\diff \LF_{11}, \diff \LF_{12}, ..., \diff \LF_{1N}, \diff \LF_{21}, ..., \diff \LF_{N1}, ... \diff \LF_{NN})$ of $N \times N$ variables $ (L_{11}, L_{12}, ..., L_{1N}, L_{21}, ..., L_{N1}, ... L_{NN})$. Since 
\begin{align}
    \frac{\diff L^ {-1}}{\diff L_{ij}}
    &= - L^{-1} \cdot \frac{\diff L }{\diff L_{ij}}\cdot L^{-1} \nonumber \\
    &= - [L^{-1}_{ki} L^{-1}_{jl}]_{kl},
\end{align} 
the Jacobian matrix of $\LF$ is given by an $N \times N$ by $N \times N$ matrix
\begin{align}
    \diff^2 \LF = &-\sum_{J \subseteq [N]} \hat{p}_{J}\big[(L^{-1}_{ki}L^{-1}_{jl})\mathbb{I}\{i \in J,j \in J,k\in J,l \in J \}\big]_{ij,kl} \nonumber \\
    &+ \big[ (L+I)^{-1}_{ki}(L+I)^{-1}_{jl} \big]_{ij,kl},
\end{align}

where the index $ij$ and $kl$ have order $(1,1), ...(1,N),(2,1),...(2,N),...(N,1),...(N,N).$

Then Newton-Raphson algorithm updated  at $(m+1)$-th iteration is given by
\begin{equation}\label{e.4.3} 
    L_{m+1} = L_{m} - (\diff^{2}\hat{\Phi}_{n}(L_{m}))^{-1} \cdot \diff \hat{\Phi}_{n}(L_{m}).
\end{equation}
The matrix  $\diff^{2}\hat{\Phi}_{n}(L_{m})$ is 
$N^2\times N^2$ and the computation of its inverse 
may need some time when $N$ is large.

% The time complexity of each iteration is $\mathcal{O}(N^{6}) + \sum_{i=1}^{n} (|Z_{i}|^{3} +|Z_{i}|^4) $, where $Z_{i}$ is the i-th copy of the DPP and $n$ is the number of total copies.

% Since the maximum likelihood function is $\LF$ is non-concave and has at least exponentially many saddle points \cite{brunel2017maximum}, Newton-Raphson method could easily get trapped 

%
%\subsection{Stochastic gradient descent method}

Newton-Raphson algorithm is time consuming in particular when $N$ is large since 
we need to compute the inverse of $\diff^{2}\hat{\Phi}_{n}(L_{m})$ (which is an $N^2\times N^2$ matrix) in each iteration step.  For this reason  we try another method called stochastic gradient descent method. This method depends on a step size which we
choose to be  $\eta =0.1$ as well as initial kernel  $L_{0}$. At $(m+1)$-th iteration we do simple random sampling in our data set to pick one sample $Z_{\rho }$ and update our kernel by
\begin{equation}
    L_{m+1} = L_{m} + \eta \big(({L_{{m}_{Z_{\rho}}}})^{-1} - (L_m+I)^{-1}\big).
\end{equation}
In this way, we reduce high computational burden of Newton-Raphson algorithm \eqref{e.4.3} at each iteration but the algorithm has lower convergence rate.
% than the Newton-Raphson method.

\medskip 
\noindent{\bf Simulation results}. 
In simulation study we consider two $3$ by $3$ matrix kernels $H_1, H_2$ and
one $2$ by $2$ matrix kernel $H_3$.
%$$
 %\begin{pmatrix}
 % 1 & 0.2 & 0\\
%  0.2 & 2 & 0.3 \\
%%\end{pmatrix}   \quad
%\begin{pmatrix}
 % 7 & 0 & 0\\
%%  0 & 0 & 9 
%\end{pmatrix} \quad
%$$
To see the trend of convergence, for each kernel we generate 300, 1000, 3000, 10000, and 30000 synthetic samples using the algorithm proposed in \cite{hough2006determinantal}. 
% \footnote{This means $n=3000$. You need to take $n=500, 1000, 1500, 200, 2500, 3000, 3500, 4000, 4500, 5000 $ to see if there is trend of convergence.  If you fix $n$ let the iterations $k$ in Newton-Raphson or SGD it will never converge.  }
 The simulation results show that Newton-Raphson method outperforms SGD. For the first DPP, Newton-Raphson algorithm shows a very slow trend of convergence to the true kernel $H_1$ whereas SGD only estimates the diagonal elements well. For the second independent DPP, both Newton-Raphson and SGD methods converge to the true kernel $H_2$ pretty well. For the last one, Newton-Raphson converges to a wrong critical point while SGD is numerically unstable. 
 
 Our simulation results of n = 30000 are summarized in Table \ref{table1}, which suggests learning the dependence and the repulsive behavior between items seems  beyond the scope of general numerical methods. This unsatisfactory result may be because on one hand  both algorithms involve computing the inverse of matrices at each iteration, which could yield large values and hence may not be stable. On the other hand, the likelihood function
         of a DPP is usually non-concave and has at least exponentially many saddle points \cite{brunel2017maximum}. This unsatisfactory performance of 
 the two well-known algorithms motivates us to find    
 an effective method suitable to  the maximum of a likelihood function
         of a DPP.  In the next section, this issue  
         is resolved for the two by two matrix kernel since we can obtain the 
         explicit maximizer in this case. 
         % Newton-Raphson method could easily get trapped 

\begin{center}
\begin{table}[]
\begin{tabular}{ c c c c }
\toprule
\textbf{True kernel} & \textbf{Initial kernel} & ${\textbf{Newton-Raphson}\atop k=100 \ \ \hbox{iterations}}$ & ${\textbf{SGD} \atop k=60000 \ \ \hbox{iterations}}$\\
\midrule\\
\addlinespace[-2ex]
$ \begin{pmatrix}
  1 & 0.2 & 0\\
  0.2 & 2 & 0.3 \\
  0 & 0.3 & 3 
\end{pmatrix} $ &
$ \begin{pmatrix}
  1 & 0.1 & 0\\
  0.1 & 1 & 0.1 \\
  0 & 0.1 & 1 
\end{pmatrix}$ &
 $\begin{pmatrix}
  0.994 & 0.167 & -0.048\\
  0.167 & 2.013 & 0.241 \\
  -0.048 & 0.241 & 3.029
\end{pmatrix}  $ &
$\begin{pmatrix}
  0.926 & 0.015 & 0.001\\
  0.015 & 2.323 & 0.025 \\
  0.001 & 0.025 & 3.068 
\end{pmatrix}  $\\
\addlinespace[1.5ex]
$ \begin{pmatrix}
  7 & 0 & 0\\
  0 & 5 & 0 \\
  0 & 0 & 9 
\end{pmatrix} $ &
$ \begin{pmatrix}
  1 & 0 & 0\\
  0 & 1 & 0 \\
  0 & 0 & 1 
\end{pmatrix} $ &
$ \begin{pmatrix}  7.009 & 0 & 0 \\ 0 &  5.070 & 0 \\ 0 & 0& 9.0348 \end{pmatrix}$ &
$\begin{pmatrix}
  7.002 & 0 & 0\\
  0 & 5.156 & 0 \\
  0 & 0 & 9.035
\end{pmatrix} $ \\
\addlinespace[1.5ex]

$ \begin{pmatrix}
  1 & 1\\
  1 & 2\\
\end{pmatrix} $ &
$ \begin{pmatrix}
  0.5 & 0.1\\
  0.1 & 0.5\\
\end{pmatrix} $ &
$ \begin{pmatrix}
  0.657 & 0\\
  0 & 1.499\\
\end{pmatrix} $ &
Numerically unstable

\\
\bottomrule
\end{tabular}
\caption{ We test the performance of algorithms for three different true kernels when the sample size is 30000. We obtain the result after a specific number $k$ of iterations. In the experiment, we observe that both algorithms become unstable as iteration increases to a certain certain point so we choose $k=100$ for Newton-Raphson and $k=60000$ for SGD since for the later one each iteration needs less time to complete.}
%The subscript of the matrix stands for number of iterations. }
\label{table1}
\end{table}
\end{center}

\section{Two-by-two block kernel}
Due to non-identifiability of DPPs  illustrated  in \cite[Theorem 4.1]{kulesza2012learning} that we mentioned earlier in Section 3, we have to 
use $\Tilde L_n$ defined by \eqref{e.3.16a} as our maximum likelihood estimator.   This formula involves a choice of $\hat D_n$ which may also be random.  However, the way  to choose $\hat  D_n$ is not available 
  to us in general.  To get rid of this choice of $\hat  D_n$ 
  we show in this section that if the kernels of determinantal point processes are two-by-two symmetric positive 
semi-definite matrices, we can obtain the desired results without involving the choice of $\hat  D_n$.  
%Furthermore, in this case  the maximum likelihood estimators can be obtained  analytically and will completely eliminate  the ineffectiveness of the  well-known 
%numerical algorithms.  
Our  result can also be immediately extended  to any diagonally  blocked  two by two   matrices. Since the commonly used iteration  methods presented in the previous section are not so effective in finding the maximum likelihood estimator  numerically  this explicit analytic form we obtained is also very useful in compute the estimator numerically
in practice.   On the other hand, 
 Our method effective  to  two by two matrices is difficult to apply to more general higher dimensional kernels. 

Let $Z \sim \text{DPP}(L^{\star})$, where $L^{\star} = \bigg(\begin{matrix}
	a^* & b^*\\
	b^* & c^*
\end{matrix}\bigg)$, and the ground set be $\cY = [2]$. To ensure $L^\star\in \mathcal{S}_{[2]}^{++}$, we assume $$a^*, c^* > 0$$ and $$ a^*c^* -{b^*}^{2} \geq 0.$$
We can always assume $b$ is non-negative since by identifiability of DPPs, $\bigg(\begin{matrix}
	a & b\\
	b & c
\end{matrix}\bigg)$ and $\bigg(\begin{matrix}
	a & -b\\
	-b & c
\end{matrix}\bigg)$  give the same DPP.
For ease of notation, let $\hp_{0}, \hp_{1},\hp_{2},\hp_{3}$ denote the empirical probability of the subset $\{\emptyset\}$, $\{1\}$, $\{2\}$, $\{1,2\}$ respectively and let $p_{0},p_{1}, p_{2}, p_{3}$ denote the theoretical probability respectively.
The relationship between $(a,b,c) $ and $(p_0, p_1, p_2, p_3)$ are given by
\begin{equation*} \label{relation}
	( {a}, {b}, {c}) = \Big( \frac{  p_{1}}{ p_{0}},
	\frac{\sqrt{ p_{1} p_{2}-  p_{0} p_{3}}}{ p_{0}},
	\frac{ p_{2}}{ p_{0}} \Big) ,
\end{equation*}
and inversely 
\[
\begin{split}
	p_0=& \frac{1}{(a+1)(c+1)-b^2}\,, \qquad 
	p_1=  \frac{a}{(a+1)(c+1)-b^2}\,, \\
	p_2=&\frac{c}{(a+1)(c+1)-b^2}\,,\qquad 
	p_3= \frac{ac-b^2}{(a+1)(c+1)-b^2}\,. 
\end{split}
\]	

The  likelihood function defined in \eqref{LF}  becomes now  
\begin{align} \label{LFtwo}
	\LF &= \sum_{J\in [2]} \hp_{J} \log(L_{J}) - \log \det(L+I) \nonumber \\
	&= \hp_{1} \log a +\hp_{2}\log c +\hp_{3}\log(ac-b^{2}) - \log [(a+1)(c+1)-b^{2}]
\end{align}
To find the critical point of \eqref{LFtwo}  we first let the partial derivative of $\LF$ with respect to $b$ equal zero and get
\begin{align}
	%\frac{\partial \LF}{\partial a} = & \frac{\hp_{1}}{a} %+\frac{\hp_{3}c}{ac-b^2} -\frac{c}{(a+1)(b+1)-b^2} = 0 \\
	\frac{\partial \LF}{\partial b} =& \frac{2\hp_{3}b}{ac-b^2}+ \frac{2b}{(a+1)(c+1)-b^2} =0 .
	%\frac{\partial \LF}{\partial c} = & \frac{\hp_{2}}{c} %+\frac{\hp_{3}a}{ac-b^2} -\frac{a}{(a+1)(b+1)-b^2} = 0.
\end{align}
Then we have $b$ is either equal to 0 or 
\begin{equation}
	b^{2} = \frac{ac-(a+1)(c+1)\hp_{3}}{1-\hp_{3}}. \label{b}
\end{equation}

If $b =0$, then by setting the partial derivative with respect to $a$ and $c$ to zero and notice that $\hp_{0} + \hp_{1} + \hp_{2} + \hp_{3} = 1$ we get the first critical point
\begin{equation}
	(\hat{a},\hat{b},\hat{c}) = \Bigg(\frac{\hp_{1} +\hp_{3}}{\hp_{0}+\hp_{2}}, 0,  \frac{\hp_{2}+\hp_{3}}{\hp_{0}+\hp_{1}}\Bigg). 
 \label{1critical}
\end{equation}
We remark that this is the point which Newton-Raphson method converges to in the last simulation study.
This critical point exists only if $\hp_{0} + \hp_{2}$ and $\hp_{0}+\hp_{1}$ is nonzero. Since empirical probability converges to its corresponding theoretical probability almost surely and $p_{0} >0 $, the strong law of large numbers implies the critical point exists almost surely when $n$ is sufficiently large.

If $b\not=0$, then we can use \eqref{b} to estimate $\hat b$ once $\hat a, \hat c$ are obtained:
\begin{equation}
	\hat b  = \sqrt{\frac{\hat a\hat c-(\hat a+1)(\hat c+1)\hp_{3}}{1-\hp_{3}}}. \label{hatb}
\end{equation}
To find the maximum likelihood estimators $\hat a$ and $\hat c$ of $a$ and $c$ 
we plug \eqref{b} into $\LF$ to obtain  
\begin{equation}
	\LF = \hp_{1} \log a +\hp_{2} \log c + (\hp_{3} -1)\log (a+c+1) -(\hp_{3}-1)\log\frac{\hp_{3}}{1-\hp_{3}} + \log\hp_{3}.
\end{equation}
Letting  $\frac{\partial \LF}{\partial a}$ and $\frac{\partial \LF}{\partial c}$ equal zero yields 
\begin{align}
	\begin{cases} \frac{\partial \LF}{\partial a} = 
		\frac{\hp_{1}}{a} + \frac{\hp_{3}-1}{a+c+1} = 0 \\
		\frac{\partial \LF}{\partial c} =  \frac{\hp_{2}}{c}+\frac{\hp_{3}-1}{a+c+1} =0.\\
	\end{cases} 
\end{align}
The above system of function equations can be explicitly solved  and  combining it together with \eqref{hatb} yields
\begin{equation} \label{estimator}
	(\hat{a},\hat{b},\hat{c}) = \Big( \frac{\hp_{1}}{\hp_{0}},
	\frac{\sqrt{\hp_{1}\hp_{2}-\hp_{0}\hp_{3}}}{\hp_{0}},
	\frac{\hp_{2}}{\hp_{0}} \Big) ,
\end{equation}
from which we have this critical point exists only if $\hp_{0} > 0$ and $\hp_{1}\hp_{2}-\hp_{0}\hp_{3} \geq 0$. Again by strong laws of large numbers, the second critical point also exists and converges to the true value almost surely. In fact,  we have  almost surely,
$$\frac{\hp_{1}}{\hp_{0}} \to \frac{p_{1}}{p_{0}} = a^*, \quad \frac{\sqrt{\hp_{1}\hp_{2}-\hp_{0}\hp_{3}}}{\hp_{0}} \to \frac{\sqrt{p_{1}p_{2}-p_{0}p_{3}}}{p_{0}} = b^*, \quad \frac{\hp_{2}}{\hp_{0}} \to c^* .$$
We test the estimator of the 2 by 2 kernel in earlier simulation study by taking $n = 300, 3000, 10000, 30000, $   and find the estimates respectively are
$$ \begin{pmatrix}
  0.8598 & 0.8661\\
  0.8661 & 1.9159\\
\end{pmatrix}, 
\begin{pmatrix}
  0.9454 & 0.9184\\
  0.9184 & 1.9868\\
\end{pmatrix}, 
 \begin{pmatrix}
  1.0247 & 1.0253\\
  1.0253 & 2.0241\\
\end{pmatrix}, 
\begin{pmatrix}
  0.9942 & 1.0056\\
  1.0056 & 2.0015\\
\end{pmatrix}, 
$$
which shows the estimator \eqref{estimator} is consistent numerically to the true kernel.
%\footnote{The same reason as in previous footnote. This means $n=3000$. You need to take $n=500, 1000, 1500, 200, 2500, 3000, 3500, 4000, 4500, 5000 $ to see if there is trend of convergence.    }

Furthermore, we can establish the central limit theorem for the estimator \eqref{estimator}, which corresponds to the result in Theorem \ref{CLT}.
\begin{theorem}
	Assume $ b > 0 $, then the estimator $(\hat{a},\hat{b},\hat{c})$ in {\eqref{estimator}} is asymptotically normal,
	\begin{equation}
		\sqrt{n} ( (\hat{a},\hat{b},\hat{c}) -(a^*,b^*,c^*)) \xrightarrow[n \longrightarrow \infty]{} \mathcal{N} (\bm{0},-V({a^*,b^*,c^*})),
	\end{equation}
	where the convergence holds in distribution and  $V({a^*,b^*,c^*})$ is the inverse of the Hessian matrix of the expected maximum likelihood function $ \Phi(a,b,c) = p_{1} \log a +p_{2}\log c +p_{3}\log(ac-b^{2}) - \log [(a+1)(c+1)-b^{2}] $.
\end{theorem}
\begin{proof}
	Let $Z_{1},..., Z_{n}$ be n independent subsets of  $ Z \sim \text{DPP}(L^*) $, where $L^*= \bigg(\begin{matrix}
		a^* & b^*\\
		b^* & c^*
	\end{matrix}\bigg)$. Let $X_{i}$ be the random vector $(\mathbb{I}_{ \{Z_{i} = \emptyset\} },\mathbb{I}_{ \{Z_{i} = \{1\}\} },\mathbb{I}_{ \{Z_{i} = \{2\}\} },\mathbb{I}_{ \{Z_{i} = \{1,2\}\} })^{T}$, where $\mathbb{I}_{\{\cdot\}}$ stands for the indicator random variable. Then $X_{i}$ has mean $ \bm{\mu} = (p_{0},p_{1},p_{2},p_{3})^{T}$ and covariance matrix 
	$$ \mathbf{\Sigma} = 
	\begin{pmatrix}
		p_{0} - p_{0}^{2} & -p_{0}p_{1} & -p_{0}p_{2} & -p_{0}p_{3}\\
		-p_{0}p_{1} &p_{1} - p_{1}^{2}  & -p_{1}p_{2} & -p_{1}p_{3} \\
		-p_{0}p_{2} & -p_{1}p_{2} & p_{2} - p_{2}^{2}  & -p_{2}p_{3} \\
		-p_{0}p_{3} & -p_{1}p_{3} & -p_{2}p_{3} & p_{3} - p_{3}^{2} 
	\end{pmatrix}\,. 
	$$
	By central limit theorem, $\sqrt{n}(\overline{X}_{n} - \bm{\mu})$ converges to a multivariate normal distribution with mean $\bm{0}$ and covariance $\bm{\Sigma}$. Let a function $g: \mathbb{R}^4 \to \mathbb{R}^3$ be defined by 
	$$ g(x_{1},x_{2},x_{3},x_{4}) = (\frac{x_{2}}{x_{1}},\frac{\sqrt{x_{2}x_{3}-x_{1}x_{4}}}{x_{1}},\frac{x_{3}}{x_{1}}). $$
	Its Jacobi matrix $\dot{g}(\bm{x}) = \big[\frac{\partial g_{i}}{\partial x_{j}}\big]_{3\times4}$ is given by  
	\begin{equation*}
		\begin{pmatrix}
			-\frac{x_{2}}{x_{1}^2} & \frac{1}{x_{1}}& 0 & 0\\
			-\frac{x_{4}}{2x_{1}\sqrt{x_{2}x_{3}-x_{1}x_{4}}} - \frac{\sqrt{x_{2}x_{3}-x_{1}x_{4}}}{x_{1}^2} &\frac{x_{3}}{2x_{1}\sqrt{x_{2}x_{3}-x_{1}x_{4}}}  &\frac{x_{2}}{2x_{1}\sqrt{x_{2}x_{3}-x_{1}x_{4}}}  & -\frac{1}{2\sqrt{x_{2}x_{3}-x_{1}x_{4}}} \\
			-\frac{x_{3}}{x_{1}^2}& 0& \frac{1}{x_{1}}  & 0 
		\end{pmatrix}.
	\end{equation*}
	Now we are in the position to apply Delta method \cite{van2000asymptotic} to get 
	$$\sqrt{n}\big((\hat{a},\hat{b},\hat{c})-(a^*,b^*,c^*)\big) = \sqrt{n}\big(g(\overline{X}_{n}) - g(\bm{\mu})\big)\xrightarrow[]{d} \mathcal{N}(\bm{0},\dot{g}(\bm{\mu})\bm\Sigma \dot{g}(\bm{\mu})' ). $$
	After tedious matrix computations, $\dot{g}(\bm{\mu})\bm\Sigma \dot{g}(\bm{\mu})'$ is found to be 
	\begin{equation*} 
		\hspace{-0.5in}
		D\begin{pmatrix} 
			(a^*+{a^*}^2) &  \sigma_{12}  & \sigma_{13} \\
			\sigma_{12}  & \frac{\frac{a^*c^*}{{b^*}^2}-1}{4}D + \frac{a^*+c^*+4a^*c^*}{4} & \sigma_{23} \\
			\sigma_{13}  &\sigma_{23}  & c^*+{c^*}^2 
		\end{pmatrix},
	\end{equation*} 
	where 
	\[
	\left\{\begin{split}
		D = & (a^*+1)(c^*+1) - {b^*}^2\,; \\
		\sigma_{12} =& (\frac{a^*c^*}{2b^*}+a^*b^*+\frac{a^*}{2b^*}(a^*c^*-{b^*}^2))\,; \\
		\sigma_{13} =&a^*c^*\,;\\
		\sigma_{23}=&\frac{a^*c^*}{2b^*} +b^*c^* + \frac{c^*}{2b^*}(a^*c^*-{b^*}^2)\,.\\
	\end{split}  \right.
	\]
	It is straightforward to verify the above matrix is the inverse of the Hessian matrix of the expected maximum likelihood function $\Phi(L)$, that is, $-V(a^*,b^*,c^*)$, which in turn verifies Theorem \ref{CLT} in this special case. However, in this two-by-two case, our maximum likelihood estimator is unique without the maneuver of the identifiability  \eqref{e.3.16a}.
\end{proof}

This idea can be extended to blocked ensemble with the two by two block 
submatrices. 
If $L^{\star}$ is a matrix with k two-by-two blocks $J_{1}, ..., J_{k}$
\begin{equation}
	\begin{pmatrix}
		J_1  &      &    &    \\
		&     J_2     &  & \\
		&     &        \ddots   & \\
		&      &       &       J_k       
	\end{pmatrix}=
	\begin{pmatrix}
		a_{1} & b_{1} &      &     & & \\
		b_{1} & c_{1} &      &      & & \\
		&     & a_{2} & b_{2} &      &   \\
		&      & b_{2} & c_{2} &      &   \\
		&     &       &       & \ddots    \\
		&      &       &       &      & a_{k}  & b_{k}    \\
		&       &       &       &      & b_{k}   & c_{k} 
	\end{pmatrix},
	\label{e.4.10}
\end{equation}
where for each $1\leq i \leq k$, $a_{i}$, $b_{i}$, $c_{i} > 0$ and $a_{i}c_{i} -b_{i}^{2} \geq 0$. Let ground set $\cY$ of this DPP be $\{J_{1}^{1}, J_{1}^{2},J_{2}^{1},J_{2}^{2},...,J_{k}^{1},J_{k}^{2}\}$ and for each $1\leq i \leq k$,
\begin{align}
	\hp_{J_{i}}^{0} =& \frac{1}{n} \sum_{m=1}^{n} \mathbb{I} \{ J_{i}^{1} \notin Z_{m}, J_{i}^{2} \notin Z_{m} \} \\
	\hp_{J_{i}}^{1} =& \frac{1}{n} \sum_{m=1}^{n} \mathbb{I} \{ J_{i}^{1} \in Z_{m}, J_{i}^{2} \notin Z_{m} \} \\
	\hp_{J_{i}}^{2} =& \frac{1}{n} \sum_{m=1}^{n} \mathbb{I} \{ J_{i}^{1} \notin Z_{m}, J_{i}^{2} \in Z_{m} \} \\
	\hp_{J_{i}}^{3} =& \frac{1}{n} \sum_{m=1}^{n} \mathbb{I} \{ J_{i}^{1} \in Z_{m}, J_{i}^{2} \in Z_{m} \} ,
\end{align}
where $ Z_{1}, ..., Z_{n}$ are n independent subsets drawn from $\text{DPP}(L^{\star})$.
By the construction \eqref{e.4.10} we  see that  $Z\cap J_{1}, ..., Z \cap J_{k}$ are mutually independent.
Then the result of critical point for two by two matrix can be applied:
\begin{equation}
	(\hat{a}_{i}, \hat{b}_{i},\hat{c}_{i}) =
	\Bigg(\frac{\hp_{J_{i}}^{1}}{\hp_{J_{i}}^{0}},
	\frac{\sqrt{\hp_{J_{i}}^{1}\hp_{J_{i}}^{2} -\hp_{J_{i}}^{0}\hp_{J_{i}}^{3}}}{\hp_{J_{i}}^{0}}
	, \frac{\hp_{J_{i}}^{2}}{\hp_{J_{i}}^{0}}\Bigg),
\end{equation}
for every $1 \leq i\leq k$. 

However the above  method encounters some  difficulties when the kernel has dimension higher than 2 with general form. For example, if the kernel is a $3 \times 3$ matrix
$$\begin{pmatrix}
	a & d & e\\
	d & b & f\\
	e & f & c 
\end{pmatrix},$$
the letting  the gradient of likelihood   function $\LF$ equal zero will yield 
\begin{align}
	\diff \LF =& \sum_{J \subseteq [3]} \hp_{J} L_{J}^{-1} - (L+I)^{-1} = 0 \,. \nonumber 
\end{align}
Computing $L^{-1}$ and $(L+I)^{-1}$ could  be troublesome. For example, $L^{-1}$ is:
$$\frac{1}{a(bc-f^{2}) -d(cd-ef)+e(df-be)}
\begin{pmatrix}
	bc-f^{2} & -cd+ef & -be+df\\
	-cd+ef & ac-e^{2} & de-af\\
	-be+df & de-af & ab-d^{2}
\end{pmatrix}$$
which is difficult to use to obtain explicit 
maximum likelihood estimator.

\section{Conclusion} 
In this paper,  we study the  maximum likelihood estimator for the ensemble matrix associated with    determinantal process. Brunel et al show that the expected likelihood function $\Phi(L)$ is locally strongly concave around true value $L^{\star}$ if and only if $L^\star$ is irreducible, since the Hessian matrix of $\Phi{(L)}$ at $L^ \star$ is negative definite. Then they prove the maximum likelihood estimator (MLE) is consistent with respect to   the convergence in probability and when $L^\star$ is irreducible they also obtained the central limiting theorem for the MLE.  We improve   their results to show that the MLE is also strongly consistent with respect to the  almost sure convergence. Moreover, we obtain the Berry-Esseen type result for the central limiting theorem and find the $n^{-\frac{1}{4}}$   rate of convergence of the MLE to normality. Last, we obtain the explicit form of the MLE where $L^{\star}$ is a two by two block matrix, 
 or a  block matrix  whose blocks are two
by two matrices.  This explicit form enables us to avoid the use of $\hat D_n$ in \eqref{e.3.16a} which is inconvenient in applications. The strong consistency and central limit theorem follows from these explicit forms, which demonstrates the general strong consistency and central limit theorem proved earlier. It would be interesting to find the explicit form of some particular higher dimensional DPPs. However, as the learning of maximum likelihood of DPPs is proven to be NP-hard, the explicit form for general ensembles, even if was found, would be very difficult to compute.

In addition to the maximum likelihood estimator there are also other approaches in lieu of  MLE.  Let us only  mention  one alternative  approach.
% to the ground set be $[N]$. Now we bypass the difficulty by only focusing on all the two by two principal minors of the kernel. 
For all J such that $\vert J \vert \leq 1 $, we let
\begin{equation}
	\frac{\det(L_{J})}{\det(L+I)} = \hp_{J}, \label{higherdimension}
\end{equation}
where the left hand side is the theoretical probability of $J$ and the right hand  side 
is the empirical probability of $J$.   Taking $J=\{i\}$  suggests us  the following estimator for $L_{ii}$. 
\begin{equation}
\hat 	L_{ii} = \frac{\hp_{i}}{\hp_{0}}.
\end{equation}
Using equations \eqref{higherdimension} for $\vert J \vert = 2$ again we are able to determine the off-diagonal elements up to the sign
\begin{equation}\label{e.5.3} 
	L^{2}_{ij} = \frac{\hp_{i}\hp_{j} -\hp_{\emptyset}\hp_{\{i,j\}}}{\hp^{2}_{\emptyset}}, 
\end{equation}
where $i \neq j $. Notice that this is the maximum likelihood estimator when $L$ is two dimensional. There is a question on how to choose the sign  for $L_{ij}$ in
\eqref{e.5.3}, which has been resolved by \cite{urschel2017learning} with graph theory. We shall not pursue this direction in current paper.    
 
%
%
%\section{Proofs of Section}
%
%\subsection{Proof of Proposition }
%
%\section{Technical Lemmas} 
%from \cite{affandI_2014learning}, we know....

\bibliographystyle{alpha}
\bibliography{refs}

\end{document}